\documentclass[12pt,leqno]{amsart}
\usepackage{euscript, amssymb, amsmath, amsthm}
\usepackage{epsfig}
\usepackage{graphicx}
\usepackage{caption}
\usepackage{longtable}
\usepackage{dcolumn}
\usepackage{setspace}
\usepackage[most]{tcolorbox}
\definecolor{webred}{rgb}{0.75,0,0}
\definecolor{webgreen}{rgb}{0,0.75,0}
\usepackage[citecolor=webgreen,colorlinks=true,linkcolor=blue]{hyperref}
\definecolor{refkey}{gray}{0.75}

\numberwithin{equation}{section}

\newtheorem{theo}{Theorem}[section]
\newtheorem{lem}{Lemma}[section]

\newtheorem{Def}[theo]{Definition}
\theoremstyle{remark}
\newtheorem{rem}{Remark}[section]

\newcommand{\ep}{\varepsilon}
\def\R{{\mathbb{R}}}

\def\d{\displaystyle}
\def\e{{\varepsilon}}
\def\p{\partial}
\def\G{G_1}
\def\GG{\tilde{G}_1}
\def\H{G_2}
\def\HH{\tilde{G}_2}

\date{}


\tcbset{
    frame code={}
    center title,
    left=0pt,
    right=0pt,
    top=0pt,
    bottom=0pt,
    colback=gray!10,
    colframe=white,
    width=\dimexpr\textwidth\relax,
    enlarge left by=0mm,
    boxsep=5pt,
    arc=0pt,outer arc=0pt,
    }

\begin{document}

\title[New blow-up result for  the coupled wave equations  - Invariant case]{New blow-up result for  the weakly coupled wave equations  with a scale-invariant damping and time derivative nonlinearity}
\author[M. Hamouda  and M. A. Hamza]{Makram Hamouda$^{1}$ and Mohamed Ali Hamza$^{1}$}
\address{$^{1}$ Basic Sciences Department, Deanship of Preparatory Year and Supporting Studies, P. O. Box 1982, Imam Abdulrahman Bin Faisal University, Dammam, KSA.}

\medskip

\email{mmhamouda@iau.edu.sa (M. Hamouda)} 
\email{mahamza@iau.edu.sa (M.A. Hamza)}

\pagestyle{plain}


\maketitle
\begin{abstract}

We consider in this article the weakly coupled system of wave equations  in the \textit{scale-invariant case} and with time-derivative nonlinearities. Under the usual assumption of small initial data, we obtain an improvement of the delimitation of the blow-up region by obtaining a new candidate for the critical curve. More precisely, we enhance the results obtained in \cite{Palmieri} for the system under consideration in the present work. We believe that our result is optimal in the sense that beyond the   blow-up region obtained here we may conjecture the global existence of the solution. 

\end{abstract}


\section{Introduction}
\par\quad

We consider in this article the following weakly coupled system of semilinear wave equations with damping in the {\it scale-invariant} case and nonlinearities of derivative type, namely
\begin{align}\label{G-sys}
\begin{cases}\d u_{tt}-\Delta u +\frac{\mu_1}{1+t}u_t=|\partial_t v|^p, &  x\in \mathbb{R}^N, \ t>0, \vspace{.1cm}\\
\d v_{tt}-\Delta v +\frac{\mu_2}{1+t}v_t=|\partial_t u|^q, &  x\in \R^N, \ t>0,\\
u(x,0)=\e f_1(x), \ \ v(x,0)=\e f_2(x), & x\in \mathbb{R}^N, \\ u_t(x,0)=\e g_1(x), \ \  v_t(x,0)=\varepsilon g_2(x), & x\in \mathbb{R}^N,
\end{cases}
\end{align} 
where $\mu_1,\mu_2$ are two positive constants.  Moreover, the parameter $\e$ is a positive number describing the size of the initial
data,  and $f_1,f_2,g_1$ and $g_2$ are positive functions which are compactly supported on  $B_{\R^N}(0,R), R>0$.\\
Throughout this article, we suppose that $p, q>1$.

Note that the situation, in the scale-invariant context, for coupled damped wave equations  is not a simple generalization of the case of a one damped wave equation. However,  we will point out here the recent improvements in this direction in the purpose to show the difference between the two settings (a one equation and a system). First, we recall  the Glassey conjecture which asserts that the critical power $p_G$ should be given by
\begin{equation}\label{Glassey}
p_G=p_G(N):=1+\frac{2}{N-1}.
\end{equation}
The above critical value, $p_G$, gives rise to two regions for the power $p$ ensuring the global existence (for $p>p_G$) or the nonexistence (for $p \le p_G$) of a global small data solution; see e.g. \cite{Hidano1,Hidano2,John1,Rammaha,Sideris,Tzvetkov,Zhou1}.\\
We recall here the case of one damped equation with only one time-derivative nonlinearity, namely 
\begin{equation}
\label{T-sys-bis-old}
\left\{
\begin{array}{l}
\d u_{tt}-\Delta u+\frac{\mu}{1+t}u_t=|u_t|^p, 
\quad \mbox{in}\ \R^N\times[0,\infty),\\
u(x,0)=\e f(x),\ u_t(x,0)=\e g(x), \quad  x\in\R^N.
\end{array}
\right.
\end{equation}
In the context of one damped equation, Lai and Takamura prove in \cite{LT2} a blow-up result for the solution of \eqref{T-sys-bis-old} 
and they give an upper bound of the lifespan. We stress the fact that in this case there is no restriction  for $\mu$  in the blow-up region for $p$, namely $p \in (1, p_G(N+2\mu)]$. Recently, Palmieri and Tu  proved in \cite{Palmieri}, among many other interesting results, a more accurate blow-up interval for $p$ in relationship with the solution of  \eqref{T-sys-bis-old} with a mass term. More precisely, it is proven that the solution of this problem blows up in finite time for $p \in (1, p_G(N+\sigma(\mu))]$ where 
\begin{equation}\label{sigma}
\sigma(\mu)=\left\{
\begin{array}{lll}
2 \mu & \textnormal{if} &\mu \in [0,1),\\
2 & \text{if} &\mu \in [1,2),\\
\mu & \text{if} &\mu \ge 2.
\end{array}
\right.
\end{equation}
Of course the problems studied  in \cite{Palmieri} are more general.\\

Thanks to a better understanding of the corresponding linear problem to (\ref{T-sys-bis-old}), the blow-up interval, $p \in (1, p_G(N+\sigma(\mu))]$ ($\sigma(\mu)$ is given by \eqref{sigma}),  obtained in \cite{Palmieri} and previously in \cite{LT2}, is improved in \cite{Our2}  to  reach the interval $p \in (1, p_G(N+\mu)]$, for $\mu \in (0,2)$. Our result for (\ref{T-sys-bis-old}) coincides  with the one in \cite{Palmieri}, for $\mu \ge 2$. We continue in this work to take advantage of the same technique developed in \cite{Our2} to obtain an upgrade of the blow-up region for the solution of  \eqref{G-sys}.\\

Now, going back to the system \eqref{G-sys} and let $\mu_1 = \mu_2 = 0$. Then, the system \eqref{G-sys} describes the coupling between two wave equations with time derivative nonlinearity. More precisely,  \eqref{G-sys} yields 
\begin{align}\label{G-sys-0}
\begin{cases} u_{tt}-\Delta u =|\partial_t v|^p, &  x\in \mathbb{R}^N, \ t>0,\\
 v_{tt}-\Delta v =|\partial_t u|^q, &  x\in \R^N, \ t>0,\\
u(x,0)=\e f_1(x), \ \ v(x,0)=\e f_2(x), & x\in \mathbb{R}^N, \\ u_t(x,0)=\e g_1(x), \ \  v_t(x,0)=\varepsilon g_2(x), & x\in \mathbb{R}^N.
\end{cases}
\end{align}
The study of the existence or nonexistence of  solutions to \eqref{G-sys-0} has been the subject of several works in the literature. First, let us point out the blow-up results obtained by Deng \cite{Deng} and Xu \cite{Xu}. For a family of coupled systems larger than the one studied in the present work, Ikeda et al. \cite{Ikeda-sys} have stated and proved several nice results related to different combinations of the nonlinearities in the coupled systems under examination. The context of the present article (for the damped case) is limited to the nonlinearity of derivative type (as in \eqref{G-sys}), however, other nonlinearities will be considered elsewhere; see for instance \cite{Dao-Reissig,Ikeda-sys,Palmieri1,Palmieri-Takamura-arx}. On the other hand, for the existence of solutions to  \eqref{G-sys-0}, we refer the reader to \cite{Kubo}. We notice here that, thanks to the above works, the situation is understood regarding the derivation of the curve describing the threshold between the blow-up and global existence regions for the   solutions to \eqref{G-sys-0}. More precisely, the critical (in the sense of interface between blow-up and global existence) curve for $p,q$ is given by
\begin{equation}\label{crit-curve-0}
\Upsilon(N,p,q):=\max (\Lambda(N,p,q), \Lambda(N,q,p))=0,
\end{equation}
where 
\begin{equation}\label{Lambda}
\d \Lambda(N,p,q):= \frac{p+1}{pq-1}-\frac{N-1}{2}.
\end{equation}
Under some assumptions, the solution $(u,v)$ of \eqref{G-sys-0} blows up in finite time $T(\e)$ for small initial data (of size $\e$),  namely 
\begin{equation}\label{Teps}
T(\e) \le \left\{
\begin{array}{lll}
C \e^{-\Upsilon(N,p,q)}& \text{if}&\Upsilon(N,p,q)>0,\\
\exp(C \e^{-(pq-1)})& \text{if}&\Upsilon(N,p,q)=0, \ p \neq q,\\
\exp(C \e^{-(p-1)})& \text{if}&\Upsilon(N,p,q)=0, \ p = q.
\end{array}
\right.
\end{equation}

Recently, Palmieri and Takamura \cite{Palmieri-Takamura} proved a nice result on the blow-up of the solution of a weakly coupled system of semilinear damped wave equations of derivative type. They  consider nonnegative and summable coefficients  in the damping terms (the {\it scattering} case). More precisely, using the multiplier's technique and taking advantage of the fact that the multipliers in this case are bounded, the authors in \cite{Palmieri-Takamura} prove that for the solutions of the following system:
\begin{align}\label{G-sys-b}
\begin{cases} u_{tt}-\Delta u +b_1(t)u_t=|\partial_t v|^p, &  x\in \mathbb{R}^N, \ t>0,\\
 v_{tt}-\Delta v +b_2(t)v_t=|\partial_t u|^q, &  x\in \R^N, \ t>0,\\
u(x,0)=\e f_1(x), \ \ v(x,0)=\e f_2(x), & x\in \mathbb{R}^N, \\ u_t(x,0)=\e g_1(x), \ \  v_t(x,0)=\varepsilon g_2(x), & x\in \mathbb{R}^N,
\end{cases}
\end{align}
the $(p,q)-$critical curve interestingly remains unchanged in the scattering case. Moreover, it is proven that the solutions of \eqref{G-sys-b}  blow up in finite time, and the blow-up time is  accordingly given by \eqref{Teps}. Although the result obtained in \cite{Palmieri-Takamura} is the same as for the system \eqref{G-sys-0}, but, the situation for the damped system studied in \cite{Palmieri-Takamura} is much more difficult in view of the coupling by means of two nonlinearities of derivative type. \\

In the context of the present work and to the best of our knowledge the only result on the blow-up of the weakly damped system in the scale-invariant case, \eqref{G-sys}, that we found in the literature is due to Palmieri and Tu who  proved  \cite{Palmieri}, among many other interesting results, a blow-up result for a system similar to \eqref{G-sys}. Indeed, the authors in \cite{Palmieri} studied \eqref{G-sys} in a more general setting, namely by adding two mass terms, which make the analysis somehow more delicate. 

More precisely, the authors in \cite{Palmieri} proved that there is blow-up for the system  \eqref{G-sys} for $p,q$  satisfying
\begin{equation}\label{blow-up-reg}
\Omega(N,\sigma(\mu_1),\sigma(\mu_2),p,q):=\max (\Lambda(N+\sigma(\mu_1),p,q), \Lambda(N+\sigma(\mu_2),q,p)) \ge 0,
\end{equation}
where $\Lambda$ is given by \eqref{Lambda} and $\sigma(\mu_i), i=1,2$ is given by  \eqref{sigma}.\\
Moreover, the solution $(u,v)$ of \eqref{G-sys} blows up in finite time $T(\e)$ for small initial data (of size $\e$),  namely 
\begin{equation}\label{Teps}
T(\e) \le \left\{
\begin{array}{lll}
C \e^{-\Omega(N,\sigma(\mu_1),\sigma(\mu_2),p,q)}& \text{if}&\Omega(N,\sigma(\mu_1),\sigma(\mu_2),p,q)>0,\\
\exp(C \e^{-(pq-1)})& \text{if}&\Omega(N,\sigma(\mu_1),\sigma(\mu_2),p,q)=0, \\
\exp(C \e^{-\min \left(\frac{pq-1}{p+1},\frac{pq-1}{q+1}\right)})& \text{if}&\Lambda(N+\sigma(\mu_1),p,q)=\Lambda(N+\sigma(\mu_2),q,p)=0.
\end{array}
\right.
\end{equation}

However, we will compare here our result to the one in \cite{Palmieri} by simply omitting the mass terms.

The emphasis in our  work is the study of the Cauchy problem (\ref{G-sys})  and the influence of the parameter $\mu_1, \mu_2$ on the blow-up result and the lifespan estimate. 

Therefore, thanks to a better comprehension of the role of the weak damping term in (\ref{G-sys}) in the dynamics together with a  deeper understanding of the corresponding linear problem inherited from the techniques developed in \cite{Our, Our2}, we will  improve the bound of the blow-up region.  Indeed, the result on the blow-up region, defined by \eqref{blow-up-reg} and obtained in \cite{Palmieri},   is improved here, under some assumptions to be announced in our main result, as follows:
\begin{equation}\label{blow-up-reg-imp}
\Omega(N,\mu_1,\mu_2,p,q)=\max (\Lambda(N+\mu_1,p,q), \Lambda(N+\mu_2,q,p)) \ge 0,
\end{equation}
where $\Lambda$ is given by \eqref{Lambda}.\\


The  article is organized as follows. We start in Section \ref{sec-main} by introducing  the weak formulation of (\ref{G-sys}) in the energy space. Then, we state the main theorem of our work.  In Section \ref{aux} we prove some technical lemmas that we will use, among other tools,  to conclude  the proof of the main result which is contained in Section \ref{sec-ut}.

\section{Main Result}\label{sec-main}
\par

This section is devoted to the statement of our main result. For that purpose, we first start by giving the definition of the solution of (\ref{G-sys}) in the corresponding energy space. More precisely, the weak formulation associated with  (\ref{G-sys}) reads  as follows:
\begin{Def}\label{def1}
 We say that $(u,v)$ is an energy  solution of
 (\ref{G-sys}) on $[0,T)$
if
\begin{displaymath}
\left\{
\begin{array}{l}
u,v\in \mathcal{C}([0,T),H^1(\R^N))\cap \mathcal{C}^1([0,T),L^2(\R^N)), \vspace{.1cm}\\
  \ u_t \in L^q_{loc}((0,T)\times \R^N), \ v_t \in L^p_{loc}((0,T)\times \R^N)
 \end{array}
  \right.
\end{displaymath}
satisfies, for all $\Phi, \tilde{\Phi} \in \mathcal{C}_0^{\infty}(\R^N\times[0,T))$ and all $t\in[0,T)$, the following equations:
\begin{equation}
\label{energysol2}
\begin{array}{l}
\d\int_{\R^N}u_t(x,t)\Phi(x,t)dx-\int_{\R^N}u_t(x,0)\Phi(x,0)dx \vspace{.2cm}\\
\d -\int_0^t  \int_{\R^N}u_t(x,s)\Phi_t(x,s)dx \,ds+\int_0^t  \int_{\R^N}\nabla u(x,s)\cdot\nabla\Phi(x,s) dx \,ds\vspace{.2cm}\\
\d  +\int_0^t  \int_{\R^N}\frac{\mu_1}{1+s}u_t(x,s) \Phi(x,s)dx \,ds=\int_0^t \int_{\R^N}|v_t(x,s)|^p\Phi(x,s)dx \,ds,
\end{array}
\end{equation}
and
\begin{equation}
\label{energysol3}
\begin{array}{l}
\d\int_{\R^N}v_t(x,t)\tilde{\Phi}(x,t)dx-\int_{\R^N}v_t(x,0)\tilde{\Phi}(x,0)dx \vspace{.2cm}\\
\d -\int_0^t  \int_{\R^N}v_t(x,s)\tilde{\Phi}_t(x,s)dx \,ds+\int_0^t  \int_{\R^N}\nabla v(x,s)\cdot\nabla\tilde{\Phi}(x,s) dx \,ds\vspace{.2cm}\\
\d  +\int_0^t  \int_{\R^N}\frac{\mu_2}{1+s}v_t(x,s) \tilde{\Phi}(x,s)dx \,ds=\int_0^t \int_{\R^N}|u_t(x,s)|^q\tilde{\Phi}(x,s)dx \,ds.
\end{array}
\end{equation}
\end{Def}

The following theorem states the  main result of this article.
\begin{theo}
\label{blowup}
Let $p, q >1$ and $\mu_1, \mu_2 >0$  such that 
\begin{equation}\label{assump}
\Omega(N,\mu_1,\mu_2,p,q) \ge 0,
\end{equation}
where the expression of $\Omega$ is given by \eqref{blow-up-reg}.\\
Assume that  $f_1, f_2\in H^1(\R^N)$ and $g_1, g_2 \in L^2(\R^N)$ are non-negative functions which are compactly supported on  $B_{\R^N}(0,R)$,
and  do not vanish everywhere.
Let $(u,v)$ be an energy solution of \eqref{energysol2}-\eqref{energysol3} on $[0,T_\e)$ such that $\mbox{\rm supp}(u), {\rm supp}(v)\ \subset\{(x,t)\in\R^N\times[0,\infty): |x|\le t+R\}$. 
Then, there exists a constant $\e_0=\e_0(f_1, f_2,g_1, g_2,N,R,p,q,\mu_1,\mu_2)>0$
such that $T_\e$ verifies
\begin{equation}\label{Teps1}
T(\e) \le \left\{
\begin{array}{lll}
C \e^{-\Omega(N,\mu_1,\mu_2,p,q)}& \text{if}&\Omega(N,\mu_1,\mu_2,p,q)>0,\\
\exp(C \e^{-(pq-1)})& \text{if}&\Omega(N,\mu_1,\mu_2,p,q)=0, \\
\exp(C \e^{-\min \left(\frac{pq-1}{p+1},\frac{pq-1}{q+1}\right)})& \text{if}&\Lambda(N+\mu_1,p,q)=\Lambda(N+\mu_2,q,p)=0.
\end{array}
\right.
\end{equation}
 where $C$ is a positive constant independent of $\e$ and $0<\e\le\e_0$.
\end{theo}

\begin{rem}
We notice here that Theorem \ref{blowup} shows that the critical curve for $p,q$ is in fact a shift of the dimension by $\mu_1, \mu_2$, respectively, in  \eqref{crit-curve-0}. We believe that this new blow-up region delimitation coincides with the critical one. Of course one has to rigorously confirm this assertion by proving a global existence result in the complementary region. This will be the subject of a forthcoming work.
\end{rem}

\begin{rem}
The result in 
Theorem \ref{blowup}  holds true if we replace the linear damping term in \eqref{G-sys},  $\frac{\mu_1}{1+t} u_t$ (resp. $\frac{\mu_2}{1+t} v_t$), by $b_1(t)u_t$ (resp. $b_2(t)v_t$) with $[b_i(t)-\mu_i (1+t)^{-1}]; i=1,2,$ belongs to $L^1(0,\infty)$. The proof of this  generalized damping case can be obtained by following the same steps as in the proof of Theorem \ref{blowup}   with the necessary modifications.
\end{rem}

\begin{rem}
Note that the blow-up result \eqref{Teps1} in Theorem \ref{blowup} remains true for the solutions of the system \eqref{G-sys-b} with e.g. $b_1(t)u_t=\frac{\mu_1}{1+t} u_t$ and $b_2(t)v_t$ is such that $[b_2(t)-\mu_2(1+t)^{-1}]$ belongs to $L^1(0,\infty)$. It suffices to set in this case $\mu_2=0$  in \eqref{Teps1}. The proof can be carried out  by combining the computations in the present article and  \cite{Palmieri-Takamura}.
\end{rem}

\section{Some auxiliary results}\label{aux}
\par

We define the following positive test function  
\begin{equation}
\label{test11}
\psi_i(x,t):=\rho_i(t)\phi(x), \ i=1,2, 
\end{equation}
where
\begin{equation}
\label{test12}
\phi(x):=
\left\{
\begin{array}{ll}
\d\int_{S^{N-1}}e^{x\cdot\omega}d\omega & \mbox{for}\ N\ge2,\vspace{.2cm}\\
e^x+e^{-x} & \mbox{for}\  N=1.
\end{array}
\right.
\end{equation}
Note that the function $\phi(x)$ is introduced in \cite{YZ06}   and $\rho_i(t)$, \cite{Palmieri1,Palmieri-Tu,Tu-Lin1,Tu-Lin},   is solution of 
\begin{equation}\label{lambda}
\frac{d^2 \rho_i(t)}{dt^2}-\rho_i(t)-\frac{d}{dt}\left(\frac{\mu_i}{1+t}\rho_i(t)\right)=0, \ i=1,2.
\end{equation}
The expression of  $\rho_i(t)$ reads as follows (see the Appendix for more details):
\begin{equation}\label{lmabdaK}
\rho_i(t)=(t+1)^{\frac{\mu_i+1}{2}}K_{\frac{\mu_i-1}2}(t+1), \ i=1,2,
\end{equation}
where 
$$K_{\nu}(t)=\int_0^\infty\exp(-t\cosh \zeta)\cosh(\nu \zeta)d\zeta,\ \nu\in \mathbb{R}.$$
Moreover, the function $\phi(x)$ verifies
\begin{equation*}\label{phi-pp}
\Delta\phi=\phi.
\end{equation*}
Note that the function $\psi_i(x, t)$ satisfies the corresponding conjugate equation, namely we have
\begin{equation}\label{lambda-eq}
\partial^2_t \psi_i(x, t)-\Delta \psi_i(x, t) -\frac{\partial}{\partial t}\left(\frac{\mu_i}{1+t}\psi_i(x, t)\right)=0.
\end{equation}

Throughout this article, we will denote by $C$  a generic positive constant which may depend on the data ($p,q,\mu_i,N,R,f_i,g_i$)$_{i=1,2}$ but not on $\ep$ and whose the value may change from line to line. Nevertheless, we will precise the dependence of the constant $C$ on the parameters of the problem when it is necessary.\\

The following lemma  holds true for the function $\psi_i(x, t)$.
\begin{lem}[\cite{YZ06}]
\label{lem1} Let  $r>1$.
Then, there exists a constant $C=C(N,R,r)>0$ such that
\begin{equation}
\label{psi}
\int_{|x|\leq t+R}\Big(\psi_i(x,t)\Big)^{r}dx
\leq C\rho_i^r(t)e^{rt}(1+t)^{\frac{(2-r)(N-1)}{2}}, \ i=1,2, 
\quad\forall \ t\ge0.
\end{equation}
\end{lem}

Now, we introduce the following functionals.
\begin{equation}
\label{F1def-old}
F_1(t):=e^{-t}\int_{\R^N}u(x, t)\phi(x)dx, \quad F_2(t):=e^{-t}\int_{\R^N}v(x, t)\phi(x)dx,
\end{equation}
and
\begin{equation}
\label{F2def-old}
\tilde{F}_1(t):=e^{-t}\int_{\R^N}\p_tu(x, t)\phi(x)dx, \quad \tilde{F}_2(t):=e^{-t}\int_{\R^N}\p_tv(x, t)\phi(x)dx.
\end{equation}
The next two lemmas give the first  lower bounds for $F_i(t)$ and $\tilde{F}_i(t)$, \ i=1,2, respectively.
\begin{lem}
\label{F1-old}
Assume that the assumptions in Theorem \ref{blowup} hold. Then, for $ i=1,2$, we have
\begin{equation}
\label{F1postive-old}
F_i(t)\ge \frac{\e}{2 m_i(t)}\int_{\R^N}f_i(x)\phi(x)dx, 
\quad\text{for all}\ t \in [0,T),
\end{equation}
where $m_i(t):=(1+t)^{\mu_i}$.
\end{lem}
\begin{lem}
\label{F11-old}
Under the same assumptions of Theorem \ref{blowup}, it holds that
\begin{equation}
\label{F2postive-old}
\tilde{F}_i(t)\ge \frac{\e}{2m_i(t)}\int_{\R^N}g_i(x)\phi(x)dx, \ i=1,2, 
\quad\text{for all}\ t \in [0,T).
\end{equation}
\end{lem}
The proofs of the two above lemmas are based on the multiplier's technique introduced in \cite{LT3} and used in several works; see e.g.  \cite{Our,Our2,LT,LT2,Palmieri-Takamura}. Note that the multipliers $(m_i(t))_{i=1,2}$  are not bounded.
\begin{rem}
We notice that the lower bounds obtained  in Lemmas \ref{F1-old} and \ref{F11-old} are not optimal, however, the results there are sufficient  since we only need the positivity of $F_i(t)$ and $\tilde{F}_i(t), \ i=1,2$, for all $t>0$, in the proof of our main result. Nevertheless, we will instead introduce new functionals for which we aim to obtain better lower bounds. This will be the subject of the next two lemmas.
\end{rem}

\par
We define now the functionals that we will use to prove the blow-up criteria later on:
\begin{equation}
\label{F1def}
\G (t):=\int_{\R^N}u(x, t)\psi_1(x, t)dx, \quad \H (t):=\int_{\R^N}v(x, t)\psi_2(x, t)dx,
\end{equation}
and
\begin{equation}
\label{F2def}
\GG (t):=\int_{\R^N}\p_tu(x, t)\psi_1(x, t)dx, \quad \HH (t):=\int_{\R^N}\p_t v(x, t)\psi_2(x, t)dx.
\end{equation}
The next two lemmas give the first  lower bounds for $G_i(t)$ and $\tilde{G}_i(t)$ ($i=1,2$), respectively. More precisely, we will prove that the functions $G_i(t)$ and $\tilde{G}_i(t)$ are  {\it coercive}. This is the first observation which will be used to improve the main result of this article. Indeed, in comparison with \eqref{F1postive-old} and \eqref{F2postive-old},  we obtain in the subsequent better lower bounds for the functionals $G_1(t)$ and $\tilde{G}_1(t)$; see \eqref{F1postive} and \eqref{F2postive} below, passing thus from the size $\e / (1+t)^{\mu_i/2}$ to the size $\e$.

 Although the techniques used here are similar to the ones in our previous work \cite{Our2} in the case of a one single equation, but, the situation is slightly different for the system \eqref{G-sys}. So, we will include  all the details about the proofs of the next two lemmas. However, we will only show the proofs for the solution $u$, and for $v$ the computations follow similarly.  
\begin{lem}
\label{F1}
Assume that the assumptions in Theorem \ref{blowup} hold. Let $(u,v)$ be an energy solution of \eqref{energysol2}-\eqref{energysol3}. Then, for $i=1,2$, there exists $T_0=T_0(\mu_1, \mu_2)>1$ such that 
\begin{equation}
\label{F1postive}
G_i(t)\ge C_{G_i}\, \e, 
\quad\text{for all}\ t \ge T_0,
\end{equation}
where $C_{G_i}$ is a positive constant which depends on $f_i$, $g_i$, $N,R$ and $\mu_i$.
\end{lem}
\begin{proof} 
As mentioned before, we will prove the lemma for $u$. The proof for $v$ is similar. \\
Let $t \in [0,T)$, then using Definition \ref{def1} and  performing an integration by parts in space in the fourth term in the left-hand side of \eqref{energysol2}, we obtain
\begin{equation}\label{eq3}
\begin{array}{l}
\d \int_{\R^N}u_t(x,t)\Phi(x,t)dx
-\e\int_{\R^N}g_1(x)\Phi(x,0)dx \vspace{.2cm}\\
\d-\int_0^t\int_{\R^N}\left\{
u_t(x,s)\Phi_t(x,s)+u(x,s)\Delta\Phi(x,s)\right\}dx \, ds +\int_0^t  \int_{\R^N}\frac{\mu_1}{1+s}u_t(x,s) \Phi(x,s)dx \,ds\vspace{.2cm}\\
\d=\int_0^t\int_{\R^N}|v_t(x,s)|^p\Phi(x,s)dx \, ds, \quad \forall \ \Phi\in \mathcal{C}_0^{\infty}(\R^N\times[0,T)).
\end{array}
\end{equation}
Now, substituting in \eqref{eq3} $\Phi(x, t)$ by $\psi_1(x, t)$, we infer that
\begin{equation}
\begin{array}{l}\label{eq4}
\d \int_{\R^N}u_t(x,t)\psi_1(x,t)dx
-\e\int_{\R^N}g_1(x)\psi_1(x,0)dx \vspace{.2cm}\\
\d-\int_0^t\int_{\R^N}\left\{
u_t(x,s) \partial_t \psi_{1}(x,s)+u(x,s)\Delta\psi_1(x,s)\right\}dx \, ds +\int_0^t  \int_{\R^N}\frac{\mu_1}{1+s}u_t(x,s) \psi_1(x,s)dx \,ds\vspace{.2cm}\\
\d=\int_0^t\int_{\R^N}|v_t(x,s)|^p\psi_1(x,s)dx \, ds.
\end{array}
\end{equation}
Performing an integration by parts for the first and third terms in the second line of \eqref{eq4} and utilizing \eqref{test11} and \eqref{lambda-eq}, we obtain
\begin{equation}
\begin{array}{l}\label{eq5}
\d \int_{\R^N}\big[u_t(x,t)\psi_1(x,t)- u(x,t)\partial_t \psi_{1}(x,t)+\frac{\mu_1}{1+t}u(x,t) \psi_1(x,t)\big]dx
\vspace{.2cm}\\
\d=\int_0^t\int_{\R^N}|v_t(x,s)|^p\psi_1(x,s)dx \, ds 
+\d \e \, C(f_1,g_1),
\end{array}
\end{equation}
where 
\begin{equation}\label{Cfg}
C(f_1,g_1):=\rho_1(0)\int_{\R^N}\big[\big(\mu_1-\frac{\rho_1'(0)}{\rho_1(0)}\big)f_1(x)\phi(x)+g_1(x)\phi(x)\big]dx.
\end{equation}
We notice  that the constant $C(f_1,g_1)$ is positive thanks to \eqref{lambda'lambda} and the fact that the function $K_{\nu}(t)$ is positive (see \eqref{Kmu} in the Appendix). \\
Hence, using the definition of $G_1$, as in \eqref{F1def},  and \eqref{test11}, the equation  \eqref{eq5} yields
\begin{equation}
\begin{array}{l}\label{eq6}
\d G_1'(t)+\Gamma_1(t)G_1(t)=\int_0^t\int_{\R^N}|v_t(x,s)|^p\psi_1(x,s)dx \, ds +\e \, C(f_1,g_1),
\end{array}
\end{equation}
where 
\begin{equation}\label{gamma}
\Gamma_1(t):=\frac{\mu_1}{1+t}-2\frac{\rho_1'(t)}{\rho_1(t)}.
\end{equation}
Multiplying  \eqref{eq6} by $\frac{(1+t)^{\mu_1}}{\rho_1^2(t)}$ and integrating over $(0,t)$, we deduce  that
\begin{align}\label{est-G1}
 G_1(t)
\ge G_1(0)\frac{\rho_1^2(t)}{(1+t)^{\mu_1}}+{\e}C(f_1,g_1)\frac{\rho_1^2(t)}{(1+t)^{\mu_1}}\int_0^t\frac{(1+s)^{\mu_1}}{\rho_1^2(s)}ds.
\end{align}
Using \eqref{lmabdaK} and  the fact that $G_1(0)>0$, the estimate \eqref{est-G1} yields
\begin{align}\label{est-G1-1}
 G_1(t)
\ge {\e}C(f_1,g_1)(1+t)K^2_{\frac{\mu_1-1}2}(t+1)\int^t_{t/2}\frac{1}{(1+s)K^2_{\frac{\mu_1-1}2}(s+1)}ds.
\end{align}
From \eqref{Kmu}, we have the existence of $T^1_0=T^1_0(\mu_1)>1$ such that 
\begin{align}\label{est-double}
(1+t)K^2_{\frac{\mu_1-1}2}(t+1)>\frac{\pi}{4} e^{-2(t+1)} \quad \text{and}  \quad (1+t)^{-1}K^{-2}_{\frac{\mu_1-1}2}(t+1)>\frac{1}{\pi} e^{2(t+1)}, \ \forall \ t \ge T^1_0/2.
\end{align}
Hence, we have
\begin{align}\label{est-G1-2}
 G_1(t)
\ge \frac{\e}{4}C(f_1,g_1)e^{-2t}\int^t_{t/2}e^{2s}ds\ge \frac{\e}{8}C(f_1,g_1)e^{-2t}(e^{2t}-e^{t}), \ \forall \ t \ge T^1_0.
\end{align}
Finally, using $e^{2t}>2e^{t}, \forall \ t \ge 1$, we deduce that
\begin{align}\label{est-G1-3}
 G_1(t)
\ge \frac{\e}{16}C(f_1,g_1), \ \forall \ t \ge T^1_0.
\end{align}
Hence, similarly for $v$ we have the existence of $T^2_0=T^2_0(\mu_2)>1$. Finally, to conclude we take $T_0=\max(T^1_0,T^2_0)$. This ends the proof of Lemma \ref{F1}.
\end{proof}

Now we are in position to prove  the following lemma.
\begin{lem}\label{F11}
Assume that the assumptions in Theorem \ref{blowup} hold. Let $(u,v)$ be an energy solution of \eqref{energysol2}-\eqref{energysol3}. Then, for $i=1,2$, there exists $T_1=T_1(\mu_1, \mu_2)>1$ such that
\begin{equation}
\label{F2postive}
\tilde{G}_i(t)\ge C_{\tilde{G}_i}\, \e, 
\quad\text{for all}\ t  \ge  T_1,
\end{equation}
where $C_{\tilde{G}_i}$ is a positive constant which depends on $f_i$, $g_i$, $N, R$ and $\mu_i$.
\end{lem}
\begin{proof}
As  before, we will prove the lemma for $u$. The proof for $v$ is similar. \\
Let $t \in [0,T)$, then  using the definition of $G_1$ and  $\tilde{G}_1$, given respectively by \eqref{F1def} and  \eqref{F2def}, \eqref{test11} and the fact that
 \begin{equation}\label{def23}\d G_1'(t) -\frac{\rho_1'(t)}{\rho_1(t)}G_1(t)= \tilde{G}_1(t),\end{equation}
 the equation  \eqref{eq6} yields
\begin{equation}
\begin{array}{l}\label{eq5bis}
\d \tilde{G}_1(t)+\left(\frac{\mu_1}{1+t}-\frac{\rho_1'(t)}{\rho_1(t)}\right)G_1(t)\\
=\d \int_0^t\int_{\R^N}|v_t(x,s)|^p\psi_1(x,s)dx \, ds +\e \, C(f_1,g_1).
\end{array}
\end{equation}
Differentiating the  equation \eqref{eq5bis} in time, we get
\begin{equation}\label{F1+bis}
\begin{array}{l}
\d \tilde{G}_1'(t)+\left(\frac{\mu_1}{1+t}-\frac{\rho_1'(t)}{\rho_1(t)}\right)G'_1(t)-\left(\frac{\mu_1}{(1+t)^2}+\frac{\rho_1''(t)\rho_1(t)-(\rho_1'(t))^2}{\rho_1^2(t)}\right)G_1(t) \\
\d = \int_{\R^N}|v_t(x,t)|^p\psi_1(x,t)dx.
\end{array}
\end{equation}
Using  \eqref{lambda} and   \eqref{def23}, the identity \eqref{F1+bis} becomes
\begin{align}\label{F1+bis2}
\d \tilde{G}_1'(t)+\left(\frac{\mu_1}{1+t}-\frac{\rho_1'(t)}{\rho_1(t)}\right)\tilde{G}_1(t)-G_1(t) = \int_{\R^N}|v_t(x,t)|^p\psi_1(x,t)dx.
\end{align}
Remember the definition of $\Gamma_1(t)$, given by \eqref{gamma}, we obtain 
\begin{equation}\label{G2+bis3}
\begin{array}{c}
\d \tilde{G}_1'(t)+\frac{3\Gamma_1(t)}{4}\tilde{G}_1(t)\ge\Sigma_1^1(t)+\Sigma_1^2(t)+\Sigma_1^3(t),
\end{array}
\end{equation}
where 
\begin{equation}\label{sigma1-exp}
\begin{array}{rl}
\Sigma_1^1(t):=&\d \left(-\frac{\rho_1'(t)}{2\rho_1(t)}-\frac{\mu_1}{4(1+t)}\right)\left(\tilde{G}_1(t)+\left(\frac{\mu_1}{1+t}-\frac{\rho_1'(t)}{\rho_1(t)}\right)G_1(t)\right),
\end{array} 
\end{equation}
\begin{equation}\label{sigma2-exp}
\Sigma_1^2(t):=\d \left(1+\left(\frac{\rho_1'(t)}{2\rho_1(t)}+\frac{\mu_1}{4(1+t)}\right) \left(\frac{\mu_1}{1+t}-\frac{\rho_1'(t)}{\rho_1(t)}\right) \right)  G_1(t),
\end{equation}
and
\begin{equation}\label{sigma3-exp}
\Sigma_1^3(t):=  \int_{\R^N}|v_t(x,t)|^p\psi_1(x,t) dx.
\end{equation}
Now, from \eqref{eq5bis} and \eqref{lambda'lambda1}, we deduce that there exists $T^1_1=T^1_1(\mu_1) \ge T_0$ such that
\begin{equation}\label{sigma1}
\d \Sigma_1^1(t) \ge  \frac{\e}{8} \, C(f_1,g_1)+\frac{1}{4}\int_{0}^t \int_{\R^N}|v_t(x,s)|^p\psi_1(x,s)dx ds, \quad \forall \ t \ge T^1_1. 
\end{equation}
Moreover, form Lemma \ref{F1} and \eqref{lambda'lambda1}, we conclude the existence of $\tilde{T}^1_1=\tilde{T}^1_1(\mu_1) \ge T^1_1$ such that
\begin{equation}\label{sigma2}
\d \Sigma_1^2(t) \ge 0, \quad \forall \ t  \ge  \tilde{T}^1_1. 
\end{equation}
Combining \eqref{G2+bis3}, \eqref{sigma3-exp}, \eqref{sigma1} and \eqref{sigma2}, we obtain
\begin{equation}\label{G2+bis41}
\begin{array}{rcl}
\d \tilde{G}_1'(t)+\frac{3\Gamma_1(t)}{4}\tilde{G}_1(t) &\ge& \d \frac{\e}{8} \, C(f_1,g_1)+\frac{1}{4}\int_{0}^t \int_{\R^N}|v_t(x,s)|^p\psi_1(x,s)dx ds \\ &+&\d \int_{\R^N}|v_t(x,t)|^p\psi_1(x,t) dx, \quad \forall \ t  \ge  \tilde{\tilde{T}}^1_1.
\end{array}
\end{equation}
Ignoring the nonlinear terms yields
\begin{equation}\label{G2+bis4}
\begin{array}{rcl}
\d \tilde{G}_1'(t)+\frac{3\Gamma_1(t)}{4}\tilde{G}_1(t) &\ge& \d \frac{\e}{8} \, C(f_1,g_1), \quad \forall \ t  \ge  \tilde{\tilde{T}}^1_1.
\end{array}
\end{equation}
Multiplying  \eqref{G2+bis4} by $\frac{(1+t)^{3\mu_1/4}}{\rho_1^{3/2}(t)}$ and integrating over $(\tilde{\tilde{T}}^1_1,t)$, we deduce  that
\begin{align}\label{est-G111-new}
 \tilde{G}_1(t)
&\ge \tilde{G}_1(\tilde{\tilde{T}}^1_1)\frac{(1+\tilde{\tilde{T}}^1_1)^{3\mu_1/4}}{\rho_1^{3/2}(\tilde{\tilde{T}}^1_1)}\frac{\rho_1^{3/2}(t)}{(1+t)^{3\mu_1/4}}\\&+\frac{\e}{8} \, C(f_1,g_1)\frac{\rho_1^{3/2}(t)}{(1+t)^{3\mu_1/4}}\int_{\tilde{\tilde{T}}^1_1}^t\frac{(1+s)^{3\mu_1/4}}{\rho_1^{3/2}(s)}ds, \quad \forall \ t  \ge  \tilde{\tilde{T}}^1_1.\nonumber
\end{align}
Now, observe that $\tilde{G}_1(t)=\rho_1(t)e^{t}\tilde{F}_1(t)$ where $\tilde{F}_1(t)$ is given by \eqref{F2def-old}. Hence, using Lemma \ref{F11-old} we infer that $\tilde{G}_1(t) \ge 0$ for all $t \ge 0$.\\
Therefore, using the above observation and  \eqref{lmabdaK}, we deduce that 
\begin{align}\label{est-G1111}
 \tilde{G}_1(\tilde{\tilde{T}}^1_1)\frac{(1+\tilde{\tilde{T}}^1_1)^{3\mu_1/4}}{\rho_1^{3/2}(\tilde{\tilde{T}}^1_1)}\frac{\rho_1^{3/2}(t)}{(1+t)^{3\mu_1/4}}
\ge 0, \quad \forall \ t \ge 0.
\end{align}
Employing \eqref{est-double} and \eqref{est-G1111}, the estimate \eqref{est-G111-new} yields
\begin{equation}\label{est-G2-12}
 \tilde{G}_1(t)
\ge   C\,{\e}e^{-3t/2} \int^t_{t/2}e^{3s/2}ds, \quad \text{for all} \ t \ge T^1_1:=2\tilde{\tilde{T}}^1_1.
\end{equation}
Hence, we have
\begin{align}\label{est-G1-2}
 \tilde{G}_1(t)
\ge  C\,{\e}, \quad \forall \ t \ge T^1_1.
\end{align}
Hence, similarly for $v$ we have the existence of $T^2_1=T^2_1(\mu_2)>1$ and we take $T_1=\max(T^1_1,T^2_1)$. This concludes the proof of Lemma
\ref{F11}.
\end{proof}

\section{Proof of Theorem \ref{blowup}.}\label{sec-ut}

This section is devoted to the proof of Theorem \ref{blowup} which is somehow related to the obtaining  of the critical curve associated with the nonlinear  problem \eqref{G-sys}. First, we perform a better understanding of the linear  problem associated with \eqref{G-sys} and use the results  in Section \ref{aux}.  In fact, we proved in Lemma \ref{F11}  that $\tilde{G}_i(t)$ are coercive functions for $i=1,2$. This is a crucial observation that we will use to improve the blow-up result of  \eqref{G-sys}. Thanks to the observation described above and by introducing some new functionals $L_1(t)$ and $L_2(t)$ (see \eqref{L1} and \eqref{L2} below), which verify two integral inequalities similar to the ones in \cite[(25) and (26)]{Palmieri} with $\mu_1, \mu_2$ in the present work instead of $\sigma_1, \sigma_2$ in \cite{Palmieri}, we improve the blow-up result in \cite{Palmieri} for the solution of \eqref{G-sys}. The result of this work  makes the blow-up region for \eqref{G-sys} more precise. Our result for (\ref{G-sys}) enhances the corresponding one in  \cite{Palmieri} except if  $\mu_1 \ge 2$ and $\mu_2  \ge 2$ where the two results coincide.\\

Now, setting
\begin{equation}\label{L1}
L_1(t):=
\frac{1}{8}\int_{T_2}^t  \int_{\R^N}|v_t(x,s)|^p\psi_1(x,s)dx ds
+\frac{C_6 \e}{8},
\end{equation}
and
\begin{equation}\label{L2}
L_2(t):=
\frac{1}{8}\int_{T_2}^t  \int_{\R^N}|u_t(x,s)|^q\psi_2(x,s)dx ds
+\frac{C_6 \e}{8},
\end{equation}
where $C_6=\min(C(f_1,g_1),C(f_2,g_2),8C_{\tilde{G}_1},8C_{\tilde{G}_2})$ ($C_{\tilde{G}_1}$ and $C_{\tilde{G}_2}$ are defined in Lemmas \ref{F1} and \ref{F11}, respectively) and $T_2:=T_2(\mu_1,\mu_2)>T_1$ is chosen such that $\frac{1}{4}-\frac{3\Gamma_i(t)}{32}>0$ and $\Gamma_i(t)>0$, for i=1,2, for all $t \ge T_2$ (this is possible thanks to \eqref{gamma} and \eqref{lambda'lambda1}), and let
$$\mathcal{F}_i(t):= \tilde{G}_i(t)-L_i(t), \quad \forall \ i=1,2.$$
Hence, we have for $\mathcal{F}_1$,
\begin{equation}\label{G2+bis6}
\begin{array}{rcl}
\d \mathcal{F}_1'(t)+\frac{3\Gamma_1(t)}{4}\mathcal{F}_1(t) &\ge& \d \left(\frac{1}{4}-\frac{3\Gamma_1(t)}{32}\right)\int_{T_2}^t \int_{\R^N}|v_t(x,s)|^p\psi_1(x,s)dx ds\vspace{.2cm}\\ &+&  \d \frac{7}{8}\int_{\R^N}|v_t(x,t)|^p\psi_1(x,t) dx+C_6 \left(1-\frac{3\Gamma_1(t)}{32}\right) \e\\
&\ge&0, \quad \forall \ t \ge T_2.
\end{array}
\end{equation}
Multiplying  \eqref{G2+bis6} by $\frac{(1+t)^{3\mu_1/4}}{\rho_1^{3/2}(t)}$ and integrating over $(T_2,t)$, we deduce  that
\begin{align}\label{est-G111}
 \mathcal{F}_1(t)
\ge \mathcal{F}_1(T_2)\frac{(1+T_2)^{3\mu/4}}{\rho_1^{3/2}(T_2)}\frac{\rho_1^{3/2}(t)}{(1+t)^{3\mu_1/4}}, \ \forall \ t \ge T_2,
\end{align}
where $\rho_1(t)$ is defined by \eqref{lmabdaK}.\\
Therefore we have $\d \mathcal{F}_1(T_2)=\GG(T_2)-\frac{C_6 \e}{8} \ge \GG(T_2)-C_{\tilde{G}_1}\e \ge 0$ thanks to Lemma \ref{F11} and the fact that $C_6=\min(C(f_1,g_1),C(f_2,g_2),8C_{\tilde{G}_1},8C_{\tilde{G}_2}) \le 8C_{\tilde{G}_1}$. \\
Then, we have 
\begin{equation}
\label{G2-est}
\GG(t)\geq L_1(t), \ \forall \ t \ge T_2.
\end{equation}
Similarly, we have an analogous estimate for $\HH (t)$, namely
\begin{equation}
\label{G2-est-bis}
\HH(t)\geq L_2(t), \ \forall \ t \ge T_2.
\end{equation}
By H\"{o}lder's inequality and the estimates \eqref{psi} and \eqref{F2postive}, we can bound the nonlinear term as follows:
\begin{equation}\label{vt-pho}
\begin{array}{rcl}
\d \int_{\R^N}|v_t(x,t)|^p\psi_1(x,t)dx &\geq&\d \HH^p(t)\left(\int_{|x|\leq t+R}\psi_2^{\frac{p}{p-1}}(x,t)\psi_1^{\frac{-1}{p-1}}(x,t)dx\right)^{-(p-1)} \vspace{.2cm}\\ &\geq& C \HH^p(t) \rho_1(t)\rho_2^{-p}(t)e^{-(p-1)t}(1+t)^{-\frac{(N-1)(p-1)}2}.
\end{array}
\end{equation}
Using \eqref{lmabdaK} and \eqref{est-double}, we get
 \begin{equation}\label{pho-est}
 \d \rho_1(t)e^{t} \le C (1+t)^{\frac{\mu_1}{2}}, \ \forall \ t \ge T_0/2,
 \end{equation}
 and a similar estimate holds for $ \rho_2(t)$.\\
Hence, plugging  \eqref{pho-est} in \eqref{vt-pho}  yields  
\begin{equation}
\d \int_{\R^N}|v_t(x,t)|^p\psi_1(x,t)dx \geq C (1+t)^{-\frac{(N-1)}{2}(p-1)+ \frac{\mu_1}{2}-\frac{\mu_2}{2}p}\HH^p(t), \ \forall \ t \ge T_2.
\end{equation}
From the above estimate and \eqref{G2-est-bis}, we infer that
\begin{equation}
\label{inequalityfornonlinearin}
L_1'(t)\geq C (1+t)^{-\frac{(N-1)}{2}(p-1)+ \frac{\mu_1}{2}-\frac{\mu_2}{2}p}L_2^p(t), \quad \forall \ t \ge T_2.
\end{equation}
Similarly, we obtain an analogous estimate for $L_2'(t)$,
\begin{equation}
\label{inequalityfornonlinearin2}
L_2'(t)\geq C (1+t)^{-\frac{(N-1)}{2}(q-1)+ \frac{\mu_2}{2}-\frac{\mu_1}{2}q}L_1^q(t), \quad \forall \ t \ge T_2.
\end{equation}
Integrating \eqref{inequalityfornonlinearin} and \eqref{inequalityfornonlinearin2} on $(T_2, t)$, we obtain, respectively,
\begin{equation}
L_1(t)\geq \frac{C_6 \e}{8}+C_0 \int_{T_2}^t  (1+s)^{-\frac{(N-1)}{2}(p-1)+ \frac{\mu_1}{2}-\frac{\mu_2}{2}p}L_2^p(s) ds, \quad \forall \ t \ge T_2,
\end{equation}
and
\begin{equation}
L_2(t)\geq \frac{C_6 \e}{8}+C_0 \int_{T_2}^t  (1+s)^{-\frac{(N-1)}{2}(q-1)+ \frac{\mu_2}{2}-\frac{\mu_1}{2}q}L_1^q(s) ds, \quad \forall \ t \ge T_2.
\end{equation}
Using the fact that $\d \frac{1}{T_2}(T_2+s)\le 1+s \le T_2+s$ for all $s \in (T_2, t)$, because $T_2>1$, we deduce that
\begin{equation}
\label{integ-ineq}
L_1(t)\geq \frac{C_6 \e}{8}+C_1 \int_{T_2}^t  (T_2+s)^{-\frac{(N-1)}{2}(p-1)+ \frac{\mu_1}{2}-\frac{\mu_2}{2}p}L_2^p(s) ds, \quad \forall \ t \ge T_2,
\end{equation}
and
\begin{equation}
\label{integ-ineq2}
L_2(t)\geq \frac{C_6 \e}{8}+C_1 \int_{T_2}^t  (T_2+s)^{-\frac{(N-1)}{2}(q-1)+ \frac{\mu_2}{2}-\frac{\mu_1}{2}q}L_1^q(s) ds, \quad \forall \ t \ge T_2.
\end{equation}
Therefore, we end up with the same integral inequalities as in \cite{Palmieri}; here \eqref{integ-ineq} (resp. \eqref{integ-ineq2} corresponds to (25) (resp. (26)) in \cite{Palmieri}. From these two integral inequalities we mimic the same steps, line by line, in  \cite[Sections 4.2 and 4.3]{Palmieri} to prove the blow-up result with the iteration process applied to \eqref{integ-ineq}-\eqref{integ-ineq2}. However, we should take  in consideration the fact that in the present work the shift of the dimension $N$ is with  $\mu_i$ instead of $\sigma(\mu_i)$ in \cite{Palmieri}, where $\sigma(\mu_i)$ is defined by \eqref{sigma}.

\section{Appendix}
In this appendix, we will recall some properties of the function $\rho_i(t)$, for $i=1,2$, the solution of \eqref{lambda}. Hence, following the computations in \cite{Tu-Lin} (with $\eta=1)$, we can write the expression of  $\rho_i(t)$ as follows:
\begin{equation}\label{lmabdaK-A}
\rho_i(t)=(t+1)^{\frac{\mu_i+1}{2}}K_{\frac{\mu_i-1}2}(t+1), \ i=1,2,
\end{equation}
where 
$$K_{\nu}(t)=\int_0^\infty\exp(-t\cosh \zeta)\cosh(\nu \zeta)d\zeta,\ \nu\in \mathbb{R}.$$
Using the property of  $\rho_i(t)$ in the proof of Lemma 2.1 in \cite{Tu-Lin} (with $\eta=1)$, we infer that
\begin{equation}\label{lambda'lambda}
\frac{\rho_i'(t)}{\rho_i(t)}=\frac{\mu_i}{1+t}-\frac{K_{\frac{\mu_i+1}2}(t+1)}{K_{\frac{\mu_i-1}2}(t+1)}, \ i=1,2.
\end{equation}
From \cite{Gaunt}, we have the following property for the function $K_{\mu_i}(t)$, \ i=1,2,
\begin{equation}\label{Kmu}
K_{\mu_i}(t)=\sqrt{\frac{\pi}{2t}}e^{-t} (1+O(t^{-1}), \quad \text{as} \ t \to \infty.
\end{equation}
Combining \eqref{lambda'lambda} and \eqref{Kmu}, we infer that, for $\ i=1,2$,
\begin{equation}\label{lambda'lambda1}
\frac{\rho_i'(t)}{\rho_i(t)}=-1+O(t^{-1}), \quad \text{as} \ t \to \infty.
\end{equation}

Finally, we refer the reader to \cite{Erdelyi} for more details about the properties of the function $K_{\mu_i}(t)$.


\bibliographystyle{plain}

\end{document}